\documentclass[onecolumn,12pt]{article}

\usepackage{amsfonts}
\usepackage{mathrsfs}
\usepackage{amsmath}
\usepackage{epsfig}
\usepackage{graphicx}
\usepackage{multirow}
\usepackage{bigdelim}
\usepackage{amsthm}
\usepackage{diagbox}
\usepackage{bbding}
\usepackage[left=2cm,right=2cm,top=3.5cm,bottom=3.5cm]{geometry}
\usepackage[bookmarks,colorlinks,linkcolor=blue,anchorcolor=blue,citecolor=blue]{hyperref}
\newtheorem{definition}{Definition}
\newtheorem{lemma}{Lemma}
\newtheorem{theorem}{Theorem}
\newtheorem{corollary}{Corollary}

\newtheorem{assumption}{Assumption}
\newtheorem*{assumption*}{Assumption}
\newcommand{\ud}{\mathrm{d}}
\begin{document}

\setlength{\baselineskip}{6mm}

\title{A Gradually Reinforced Sample-Average-Approximation Differentiable Homotopy Method for a System of Stochastic Equations}
\author{Peixuan Li\footnote{School of Economics and Management, Southeast University. Email: lipeixuan@seu.edu.cn}
\and \,Chuangyin Dang\footnote{The Corresponding Author. Department of Systems Engineering, City University of Hong Kong. Email: mecdang@cityu.edu.hk}
\and Yang Zhan\footnote{School of Management and Engineering, Nanjing University. Email: zhanyang@nju.edu.cn}\\
}

\date{ }
\maketitle
\begin{abstract}

This paper intends to apply the sample-average-approximation (SAA) scheme to solve a system of stochastic equations (SSE), which has many applications in a variety of fields. The SAA is an effective paradigm to address risks and uncertainty in stochastic models from the perspective of Monte Carlo principle. Nonetheless, a numerical conflict arises from the sample size of SAA when one has to make a tradeoff between the accuracy of solutions and the computational cost. To alleviate this issue, we  incorporate a gradually reinforced SAA scheme into a differentiable homotopy method and develop a gradually reinforced sample-average-approximation (GRSAA) differentiable homotopy method in this paper.  By introducing a series of continuously differentiable functions of the homotopy  parameter $t$ ranging between zero and one, we establish a differentiable homotopy system, which is able to gradually increase the sample size of SAA as $t$ descends from one to zero. The set of solutions to the homotopy system contains an everywhere smooth path, which starts from an arbitrary point and ends at a solution to the SAA with any desired accuracy. The GRSAA differentiable homotopy method serves as a bridge to link the gradually reinforced SAA scheme and a differentiable homotopy method and retains the nice property of global convergence the homotopy method possesses while greatly reducing the computational cost for attaining a desired solution to the original SSE. Several numerical experiments further confirm the effectiveness and efficiency of the proposed method.

\end{abstract}

\section{Introduction}\label{sec1}
This paper is concerned with  the problem of finding an $x\in \mathbb R^n$ satisfying
\begin{align}\label{s0}\begin{array}{ll}
(SSE): & \mathbb E_\xi [f(x,\xi)]=0,
\end{array}
\end{align}
where $\xi$ is a stochastic parameter in a probability space $\Omega$, $f(\cdot,\xi):\mathbb{R}^n\rightarrow\mathbb{R}^n$ is a continuously differentiable mapping and $\mathbb E[\cdot]$ denotes the expected value over $\Omega$.  As an effective paradigm for addressing risks and uncertainty, the SSE (\ref{s0}) can be regarded as a natural extension of a system of equations (SE) to a stochastic environment and has been extensively studied and applied in a number of areas including economics, game theory, management science and engineering. An excellent review on these applications can be found in  \cite{rockafellar2009variational}. Furthermore, the SSE is closely associated with several other stochastic problems such as stochastic variational inequalities (SVI)  \cite{chen2022dynamic,chen2012stochastic,facchinei2007finite,iusem2019incremental,rockafellar2019solving}, stochastic nonlinear complementarity problems (SNCP) \cite{chen2022stochastic,ferris2013complementarity,jiang2022convergence} and stochastic optimization (SO) \cite{duchi2011adaptive,hu2020sample,Jiang2008Stochastic,na2023inequality}.  

To solve the deterministic SE, several well-known methods have been proposed in the literature, which include Newton methods,  homotopy or path-following methods and their variants. However, the convergence of  Newton methods very much depends on the starting points. The homotopy methods are a class of powerful methods for solving SEs and possess the desired property of global convergence, which play an inevitable and lasting role in various fields; see, for instances, \cite{dang2020an,eibelshauser2023logarithmic,huang2023finding,zhang2023solution}. The homotopy methods can be classified as simplicial homotopy methods and differentiable homotopy methods. The simplicial homotopy methods, dating back to the seminal paper by \cite{B1972Homotopies,Scarf1967The}, are powerful mechanisms for solving equations with highly nonlinearities or non-smooth structures, but they cannot benefit much from the differentiability of problems and can be extremely time-consuming especially when the problem sizes are large. The differentiable homotopy methods, introduced in \cite{kellogg1976constructive}, perfectly overcome the deficiency of the simplicial homotopy methods and are capable of following a smooth path.  

Over the past few decades, the differentiable homotopy methods have been substantively investigated in the literature and applied to various areas such as general equilibrium theory and game theory for solving the problems that can be reformulated as systems of differentiable equations. Two direct proofs were given in \cite{herings2000two} to show the feasibility of linear tracing procedure, which was made differentiable for computing Nash equilibria in \cite{herings2001differentiable}. A stochastic version of linear tracing procedure was established in \cite{herings2004stationary} for the computation and selection of stationary equilibria or stationary Markov perfect equilibria in stochastic games. Later, a more efficient homotopy method based on the idea of interior-point method was developed in \cite{dang2020an}. An all-solution differentiable homotopy method was proposed in \cite{judd2012finding} to find all equilibria for static and dynamic games with continuous strategies.  A convex quadratic penalty differentiable homotopy method was developed in \cite{chen2019differentiable} to compute perfect equilibria for large scale noncooperative games. A differentiable homotopy method was described in \cite{schmedders1998computing} to solve general equilibrium models with incomplete asset markets and its reliability was evidenced by numerous numerical experiments.  A  generically convergent differentiable homotopy method was constituted in \cite{herings2002computing} to compute an equilibrium for a finance version of the general equilibrium problem considered in \cite{schmedders1998computing}.   A proximal block coordinate homotopy framework was presented in \cite{wang2020efficient}  to solve large scale sparse least squares problems through numerically following a piecewise smooth path. A polyhedral homotopy-baed method was designed in \cite{lee2023polyhedral} to find solutions to generalized Nash equilibrium problems. More differentiable homotopy methods and their applications can be found in the literature such as \cite{herings2010homotopy,kubler2000computing,Li2020,zhan2020smooth,zhou2014smoothing} and the references therein.

Unfortunately, all the existing methods fail to directly solve a stochastic system since the underlying mapping $\mathbb E_\xi [f(x,\xi)]$ of an SSE is in a form of an expected value with respect to a certain stochastic parameter vector $\xi$ and needs to be evaluated first. As demonstrated in \cite{Jiang2008Stochastic}, the evaluation of $\mathbb E_\xi [f(x,\xi)]$ is generally a tough job, because the distribution of the stochastic parameter is usually unknown and can only be simulated with historical data.  Even if the distribution is given,  computing multidimensional integrals is very costly. To deal with these difficulties, two classes of methods have been proposed in the literature: stochastic approximation (SA) methods and simulation based approaches.  It was proved in \cite{Pflug1996Optimization} that the SA method is almost surly convergent to a solution of an SSE, only when $f$ satisfies certain conditions and the samples and stepsize are suitably chosen. Notwithstanding, this method is very sensitive to the choice of the stepsize at each iteration and sometimes performs poorly in practice. A modified SA method with a better performance was proposed in \cite{nemirovski2009robust}. This better performance, however, occurs just for a special class of convex stochastic optimization and saddle point problems.  

The simulation based approaches are fashionable tools to address the uncertainty. Among them, the sample-average-approximation (SAA) is one of the most popular representatives. The basic idea of the SAA is rather simple: randomly generating  samples for the stochastic parameter $\xi$ and approximating the ``true'' underlying mapping by the average of several deterministic mappings corresponding to these samples.  As demonstrated in \cite{hu2012sample,kleywegt2002sample,shapiro1998simulation}, an appropriate incorporation of the SAA into a numerical algorithm can lead to a reasonable performance in solving general stochastic problems, and an exponential convergence rate of the SAA for the SVI and SNCP has been established under some mild conditions. Inspired by this success, this paper intends to incorporate the SAA into a differentiable homotopy method for a  solution to the SSE (\ref{s0}). To apply the SAA in the existing methods, one needs to make a tradeoff between a cheap coarse estimate and an expensive finer estimate. A common practice is to use a gradually reinforced SAA scheme, that is, initially selecting a small sample size and then gradually increasing the sample size to approximate the target expected value. This allows a rapid progress at early stages and reduces the overall compuational cost for finding a desired solution. However, this practice may fail to converge due to the changes in sample size~\cite{Byrd2012Sample}.

It is well known that a differentiable homotopy method attains global convergence through a continuous deformation process. This naturally raises the question: can we devise a differentiable homotopy method, which is able to progressively increase the sample size of SAA during the deformation process? The question seems quite challenging and were not considered in the existing work due to the intuitive collision between the discreteness of samples and the continuous deformation of the homotopy methods. To overcome this hurdle,  we introduce a sequence of continuously differentiable functions of the homotopy parameter $t$ ranging between zero and one and incorporate with these functions a gradually reinforced sample-average-approximation (GRSAA) scheme into a differentiable  homotopy method. As a result of this incorporation, we establish a differentiable homotopy system and reap a GRSAA differentiable homotopy method.  As $t$ descends from one to zero, the homotopy system gradually increases the sample size  and eventually reaches the desired SAA (\ref{ori}) at $t=0$. The solution set of the  homotopy system  contains an everywhere smooth path, which starts from an arbitrary point and ends at a solution to the desired SAA or a satisfactory approximate solution to the SSE (\ref{s0}).  The GRSAA differentiable homotopy method serves as a bridge to link a gradually reinforced SAA scheme and a differentiable homotopy method and retains the inherent property of global convergence of a standard differentiable homotopy method. 
To evince the benefit of differentiability, we also present a GRSAA simplicial homotopy method.
To make numerical comparisons with a standard differentiable homotopy method and the GRSAA simplicial homotopy method, we have implemented the GRSAA differentiable homotopy method to solve several important applications of the SSE such as the SVI and market equilibrium problems. Numerical results further verify that two main features of the GRSAA differentiable homotopy method, the gradual reinforcement in sample size and differentiability, can significantly enhance the effectiveness and efficiency.

The rest of the paper is organized as follows. In Section \ref{sec2}, we first give a brief review about the gradually reinforced SAA scheme and differentiable homotopy methods  and then develop a gradually reinforced sample-average-approximation (GRSAA) differentiable homotopy method to solve the SAA system for the SSE (\ref{s0}). The convergence properties of the proposed method are discussed in Section \ref{sec3}.  For numerical comparisons, Section \ref{sec4} describes a GRSAA simplicial homotopy method, which belongs to a type of non-differentiable homotopy method. In Section \ref{sec5}, we employ the GRSAA differentiable homotopy method to solve some numerical examples and compare the performance of the GRSAA differentiable homotopy method with that of the GRSAA simplicial homotopy method and a standard  differentiable homotopy method. We also apply the GRSAA differentiable homotopy method to solve several large-scale SSEs. All the numerical results are reported in Section \ref{sec5}. The paper is concluded in Section \ref{sec6}.


\section{A Gradually Reinforced SAA Differentiable Homotopy Method}\label{sec2}
\subsection{Background of the SAA and Homotopy Methods}\label{ssec31}
The SAA is closely associated with the sample-path optimization (SPO) method, which is an effective paradigm for solving the problems that
arise in the study of complex stochastic systems. The basic concept of the SPO method is to design some deterministic SEs with the underlying mappings being a sequence of computable functions, which has the underlying mapping of the original SSE as its limit. Some convergence conditions of the SPO method were provided in \cite{G1999Sample}.  The sample-average-approximation (SAA) is generated by formulating the underlying mappings in the SPO into some sample average mappings. More specifically, suppose that  $\xi_1,\xi_2,\ldots,\xi_N$ are independently and identically distributed samples of $\xi$, where $N$ is a positive integer. According to the Monte Carlo principle, the expected value $\mathbb E_\xi [f(x,\xi)]$ can be approximated by the deterministic mapping, $\dfrac{1}{N}\sum\limits_{i=1}^{N}f(x,\xi_i)$, and the well-known strong Law of Large Numbers assures that for any $\epsilon>0$, there exists a sufficiently large number $N_0$ such that when $N\ge N_0$,
\begin{equation}\label{monte}
\begin{array}{l}
P\{|\dfrac{1}{N}\sum\limits_{i=1}^{N}f(x,\xi_i)-\mathbb E_\xi [f(x,\xi)]|<\epsilon\}=1.
\end{array}
\end{equation}
Let $x_N\in \mathbb R^n$ be a solution to the following  SAA with very large value of $N$, 
\begin{align}\label{ori}
\dfrac{1}{N}\sum\limits_{i=1}^{N}f(x,\xi_i)=0.
\end{align}
Substituting $x_N$ into (\ref{monte}), we have $P\{|\mathbb E_\xi f(x_N,\xi)|<\epsilon\}=1$, which implies that $x_N$ provides a satisfactory  approximate solution to the SSE (\ref{s0}) when $N$ is sufficiently large. Furthermore, when $N$ goes to infinity, $x_N$ provides an accurate solution to the SSE (\ref{s0}) with probability one.  However,  the rate at which $\dfrac{1}{N}\sum\limits_{i=1}^{N}f(x,\xi_i)$ converges to  $\mathbb E[f(x,\xi)]$  is $O(N^{-\frac{1}{2}})$, that is, $N$ should be increased by 100 times in order to improve the accuracy of an estimate by one digit~\cite{kleywegt2002sample,shapiro2000rate}. Therefore,  to have a highly accurate final solution, one has to choose a very large $N$. 

As stated in Section \ref{sec1}, a differentiable homotopy method provides a globally convergent solution approach to a deterministic SE. Such a method starts from the unique solution of a trivial problem, follows a smooth path in the solution set of a homotopy system and ends at a solution to the targeted problem. One straightforward 
approach to solving the SAA  (\ref{ori}) is to apply a standard differentiable homotopy method and obtain a homotopy system,
\[(1-t)\dfrac{1}{N}\sum\limits_{i=1}^{N}f(x,\xi_i)+t(x-x^0)=0,\]
where $x^0$ is any given point and $t$ is the homotopy parameter ranging between zero and one.  Clearly, the sample size remains to be $N$ in the above system for any value of $t\in [0,1)$ and the approach has no improvement to the classic SAA paradigm. Therefore, one would like to incorporate a differentiable homotopy method with a gradually reinforced SAA scheme so that the sample size can be increased progressively as $t$ descends from one to zero. For example, we can divide the interval $[0,1]$ into $N$ subintervals, $[0,1/N)$, $[1/N,2/N)$, $\dots,$ $[1-1/N,1]$. Let $k$ be an integer smaller than $N$. At $t=1-k/N$, only the first $k$ samples are used, and we are interested in the problem $\sum\limits_{i=1}^k f(x,\xi_i)/k=0$. As $t$ decreases from  $1-k/N$ to $1-(k+1)/N$, $\xi_{k+1}$ enters the picture, and the homotopy system varies from $\sum\limits_{i=1}^{k}f(x,\xi_i)/k=0$ to $\sum\limits_{i=1}^{k+1}f(x,\xi_i)/(k+1)=0$. In this way, we do not need to use a large number of samples for every value of $t\in [0,1)$. 

However, there is a natural conflict between the SAA and the homotopy process. The homotopy process is a continuous deformation process, while the sample set for the SAA is discrete. As the homotopy  parameter $t$ varies, the samples enter the homotopy system one by one, or more generally group by group. In the above example, it is natural to construct a series of homotopies for $t$ in different subintervals $[(k-1)/N,k/N]$, $k=1,2,\dots,N$, and switch homotopies at each connection point $k/N$, which leads to a piecewise smooth path. The idea of switching homotopies was used in the literature for computing market equilibria in incomplete markets \cite{Brown1996aComputing}. Since $N$ is very large, switching homotopies will frequently occur, which makes the computation very costly.  In the next subsection, we show that this difficulty can be overcome with a sequence of continuously differentiable functions of $t$.

\subsection{A Gradually Reinforced SAA Differentiable Homotopy System}\label{ssec22}
In this subsection, we  incorporate a gradually reinforced SAA scheme into a differentiable homotopy method and develop a gradually reinforced sample-average-approximation (GRSAA) differentiable homotopy method.  In this paper,  the homotopy  parameter will descend from one to zero.  With the proposed method, the sample size in the homotopy system is getting larger and larger  as the homotopy  parameter $t$ is decreasing and becomes identical to that of the desired SAA as $t=0$. The formulation of the GRSAA  differentiable homotopy system in our method consists of four steps.
\begin{itemize}
	\item \textbf{Construction of a sequence of sample-average mappings.} 
	
Let $L$ be a positive integer with $L\le N$ and $\mathcal{L}=\{1,2,\dots, L\}$. We partition the set of $N$ samples into $L$ subsets.  For $\ell\in\cal{L}$, let $q_\ell$ be the number of samples in the first $\ell$ subsets. Clearly,  $0<q_1<q_2<\dots<q_L=N$. For each $\ell$, we define
\[f^{\ell}(x)=\frac{1}{q_\ell}\sum_{i=1}^{q_\ell}f(x,\xi_i).
\]
Let $f^0(x)=0\in\mathbb{R}^n$ for all $x\in\mathbb{R}^n$. Then $f^{\ell}(x)$, $\ell\in\mathcal{L}\cup\{0\}$, form a sequence of sample-average mappings. 

\item \textbf{Formulation of a continuous mapping deforming from $f^{\ell-1}(x)$ to $f^{\ell}(x)$ for each $\ell\in\mathcal{L}$.} 

Let $t_\ell$, $\ell\in \mathcal{L}\cup\{0\}$, be a sequence with $1=t_0>t_1>\dots>t_L=0$. Let $d^\ell(x,t)$ be a mapping on $\mathbb{R}^n\times [t_\ell,t_{\ell-1}]$ that deforms from $f^{\ell-1}(x)$ to $f^{\ell}(x)$ as $t$ decreases from $t_{\ell-1}$ to $t_\ell$. It is convenient to define
\[d^\ell(x,t)=(1-\theta_{\ell}(t))f^{\ell-1}(x) + \theta_\ell(t) f^{\ell}(x),
\]
where $\theta_{\ell}:[t_{\ell},t_{\ell-1}]\to [0,1]$ is a continuous function with $\theta_\ell(t_{\ell-1})=0$ and $\theta_\ell(t_{\ell})=1$. 

\item \textbf{Composition of a continuously differentiable mapping deforming from  $f^{0}(x)$ to $f^{L}(x)$.}

It is straightforward to define a mapping $d: \mathbb{R}^n\times[0,1]\to \mathbb{R}^n$ as $d(x,t)=d^\ell(x,t)$ when $t\in[t_{\ell},t_{\ell-1}]$, $\ell\in \mathcal{L}$. Clearly, $d(x,t)$ is a continuous mapping on $\mathbb{R}^n\times [0,1]$, which deforms from $f^{0}(x)$ to $f^{L}(x)$ as $t$ descends from $t_0 =1$ to $t_L =0$. To ensure continuous differentiability of $d(x,t)$ on  $\mathbb{R}^n\times [0,1]$, the function $\theta_\ell(t)$ should be continuously differentiable on $[t_\ell,t_{\ell-1}]$ with $\frac{\ud }{\ud t}\theta_\ell(t_\ell)=0$ and $\frac{\ud }{\ud t}\theta_\ell(t_{\ell-1})=0$. One potential candidate of $\theta_\ell(t)$ is given by
\begin{equation}
\theta_\ell(t)=\sin^2\left(\frac{t-t_{\ell-1}}{t_{\ell}-t_{\ell-1}}\frac{\pi}{2}\right).
\end{equation}
In our development, we will adopt the sequence of functions $\theta_\ell(t)$ for all $\ell\in\mathcal L$. 

\item \textbf{Establishment of a GRSAA differentiable homotopy system.}

In order to establish a GRSAA differentiable homotopy system with a unique known solution on the level of $t=1$, we incorporate with the homotopy  parameter $t$ an affinely linear function into the homotopy mapping $d(x,t)$ and arrive at a homotopy system,\begin{equation}\label{s1}
\begin{array}{l}
h(x,t)=(1-t)d(x,t)+t(x-x^0)=0,
\end{array}
\end{equation}
where $x^0\in  \mathbb R^n$ is an arbitrarily given point. 
Clearly,  at $t=t_0=1$, the homotopy system (\ref{s1}) has a unique solution $x^0$. As $t$ is decreasing, more and more samples are entering into $h(x,t)$. As $t$ goes to $t_{L}=0$, the system (\ref{s1}) approaches  the desired SAA (\ref{ori}). 
\end{itemize}
\begin{lemma}\label{l1}
	$d(x,t)$ is continuously differentiable on $\mathbb{R}^n\times [0,1]$.
\end{lemma}
\begin{proof}
	Clearly, $d(x,t)$ is continuously differentiable on $\mathbb{R}^n\times\cup_{\ell\in{\mathcal L}}(t_{\ell},t_{\ell-1})$. Next, we show that $d(x,t)$ is continuously differentiable at the connection points $(x, t_{\ell})$, $\ell\in \mathcal{L}\backslash\{L\}$.
	For each $\ell\in \mathcal{L}$, we have $\theta_{\ell}(t_{\ell-1})=0$, $\theta_{\ell}(t_{\ell})=1$ and
	\begin{equation*}
		\begin{array}{ll}
		\frac{\ud}{\ud t}\theta_\ell(t)&=\frac{\pi}{t_{\ell}-t_{\ell-1}}\sin\left(\frac{t-t_{\ell-1}}{t_{\ell}-t_{\ell-1}}\frac{\pi}{2}\right)\cos\left(\frac{t-t_{\ell-1}}{t_{\ell}-t_{\ell-1}}\frac{\pi}{2}\right)\\
\noalign{\vskip 6pt}
&=\frac{\pi}{2(t_{\ell}-t_{\ell-1})}\sin\left(\frac{t-t_{\ell-1}}{t_{\ell}-t_{\ell-1}}\pi\right).
	\end{array}
\end{equation*}
The derivatives at the connection points are given by
	\[
	\frac{\ud}{\ud t}\theta_\ell(t_{\ell-1})=0\quad\text{and}\quad \frac{\ud}{\ud t}\theta_\ell(t_\ell)=0.
	\]
	Then, for each $\ell\in \mathcal{L}\backslash\{L\}$,
	\[
	\lim\limits_{t\rightarrow t_{t_{\ell}}^-}\dfrac{\ud}{\ud t}d(x,t)=[f^{\ell+1}(x)-f^{\ell}(x)]\lim\limits_{t\rightarrow t_{t_{\ell}}^-}\dfrac{\ud}{\ud t} \theta_{\ell+1}(t)=[f^{\ell+1}(x)-f^{\ell}(x)] \frac{\ud}{\ud t}\theta_{\ell+1}(t_\ell)=0
	\]
	and\[
	\lim\limits_{t\rightarrow t_{t_{\ell}}^+}\dfrac{\ud}{\ud t}d(x,t)=[f^{\ell}(x)-f^{\ell-1}(x)]\lim\limits_{t\rightarrow t_{t_{\ell}}^+}\dfrac{\ud}{\ud t} \theta_{\ell}(t)=[f^{\ell}(x)-f^{\ell-1}(x)] \frac{\ud}{\ud t}\theta_{\ell}(t_\ell)=0.\]
	Thus, $d(x,t)$ is continuously differentiable at each connection point. This completes the proof of the lemma.
\end{proof}
From Lemma \ref{l1}, we attain the following corollary immediately.
\begin{corollary}\label{co2}
	$h(x,t)$ is a continuously differentiable mapping from $\mathbb R^n\times [0,1]$ to $\mathbb R^n$. 
\end{corollary}
Corollary \ref{l1} shows that the homotopy system (\ref{s1}) is a continuously differentiable system, which successfully settles the conflict between the discreteness of the sample set and the continuity requirement of homotopy deformation. With $h(x,t)$, we will be able to establish the existence of a smooth path, and switching homotopies is no longer needed.  
Figure \ref{f2} illustrates how the GRSAA differentiable homotopy method with $h(x,t)=0$ works.  In Figure \ref{f2}, for $t\in [0,1]$, the points on the path are solutions to the homotopy system (\ref{s1}). As a result of the continuity of $h(x,t)$, every limiting point  of the path, $x^*$, satisfies $f^{L}(x^*)=0$. We remark that a replacement  of  $t_\ell=1/(1+\tau_0\ell)$ in the above development yields a GRSAA differentiable homotopy method converging to an accurate solution to (\ref{s0}) with probability one.

\begin{figure}[!htbp]
\centering
\includegraphics[height=0.48\textwidth]{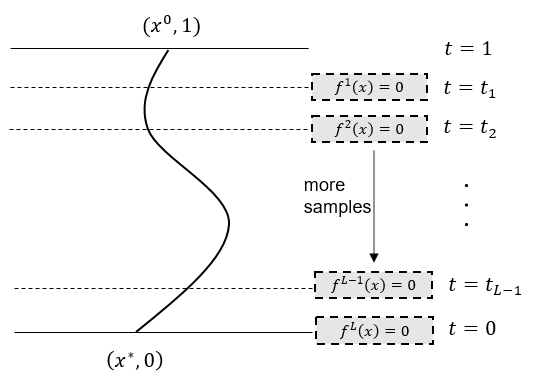}
 \caption{A GRSAA Differentiable Homotopy Method}\label{f2}
\end{figure}

\section{Development of a Smooth Path}\label{sec3}
\subsection{Convergence Properties}
 We prove in this subsection that the GRSAA differentiable homotopy method with $h(x,t)=0$ is globally convergent under some mild conditions. Let us denote by ${\rm int}(W)$ the interior of the set $W$. To ensure the existence of a smooth path that starts from $(x^0,1)$ and ends at a solution on the level of $t=0$, the path is required to be trapped in a non-empty compact set.  To meet this requirement, we make the following assumption about $f$.
\begin{assumption}\label{a1}There exist compact convex sets $\mathbb{X}\subset \mathbb R^n$ with nonempty interior such that, for any realization of $\xi$, 
$(x-x^0)^{\top}f(x,\xi)>0$ for all  $x\notin {\rm int}(\mathbb{X})$ and $x^0\in {\rm int}(\mathbb{X})$.
\end{assumption}
Assumption \ref{a1} shows a global convergence condition in fixed point problems and holds in various problem-classes including VIs, MCPs, nonlinear programming problems (NLPs), and minimax problems (MMPs) under some mild conditions; see \cite{awoniyi1983efficient,merrill1972applications,todd1976}. With this assumption, the solutions to the original problem, $f(x,\xi)=0$ for any realization of $\xi$, are restricted in a compact convex set.
Let $x^0$ be an arbitrary interior point of $\mathbb X$.
 Then, for any $x\not\in  \text{int}(\mathbb X)$, $(x-x^0)^\top d(x,t)>0$ and 
\begin{equation}\label{hxt}
\begin{array}{l}
(x-x^0)^\top h(x,t)=(1-t)(x-x^0)^\top d(x,t)+t\|x-x^0\|^2>0.
\end{array}
\end{equation}
We denote the set of solutions to the system (\ref{s1}) by 
\[H^{-1}=\{(x,t)\in \mathbb R^n\times [0,1]\,|\, h(x,t)=0\}.\]
 Under {\bf Assumption~\ref{a1}}, we prove in the following that $H^{-1}$ contains a connected component intersecting both $\mathbb{R}^n\times\{0\}$ and $\mathbb{R}^n\times\{1\}$. 
For any given $(x,t)\in \mathbb X\times [0,1]$, we constitute an unconstrained strictly convex optimization problem,
\begin{equation}\label{opt}
\begin{array}{ll}
\min\limits_{y\in\mathbb R^n}& y^\top h(x,t)+\dfrac{1}{2}\|y-x\|^2.
\end{array}
\end{equation}
An application of the sufficient and necessary optimality condition to the problem (\ref{opt}) yields the system,
\begin{align}\label{oc}
h(x,t)+(y-x)=0.
\end{align}
Let  $\varphi(x,t)$ denote the unique solution to the problem  (\ref{opt}). Clearly, $\varphi(x,t)$ is a non-empty compact convex set. It follows from the system~(\ref{oc}) that $\varphi(x,t)=x-h(x,t)$. Thus, $\varphi(x,t)$ is a continuous mapping on $\mathbb{R}^n\times [0,1]$.  
Furthermore, we derive from (\ref{hxt}) that for any given $t$, 
\begin{align}\label{as}
(x^0-x)^\top (\varphi(x,t)-x)>0,\quad\text{for all} \,\, x\not\in\text{int}(\mathbb X).
\end{align}
These results together lead us to the following theorem.
\begin{theorem}\label{th0}
Suppose that {\bf Assumption~\ref{a1}} holds. Then, for any given $t\in [0,1]$, $\varphi(x,t)$ has a fixed point in  ${\rm int}(\mathbb X)$.
\end{theorem}
\begin{proof}
Let $t$  be any given value in $[0,1]$.  For any given $x\in\mathbb X$, we denote by $\psi(x)$ the unique solution to the strictly convex quadratic optimization problem,
\[\min\limits_{y\in\mathbb X}\|y-\varphi(x,t)\|^2_2.\]
Clearly, $\psi(x)$ is a continuous mapping from $\mathbb X$ to $\mathbb X$. Since $\mathbb X$ is a convex compact set, the well-known Brouwer's fixed point theorem asserts that $\psi(x)$ has a fixed point in $\mathbb X$. 
Let $x^*(t)\in\mathbb X$ be a fixed point of $\psi(x)$. Then $x^*(t)$ is the fixed point of $\varphi(x,t)$. Next we prove this assertion by contradicton. 

Suppose that $x^*(t)\ne\varphi(x^*(t),t)$. Then we must have $x^*(t)\in\partial \mathbb X$ and $\varphi(x^*(t), t)\notin\mathbb X$. Let $\widetilde x^*(t)=(1-\rho)x^*(t)+\rho x^0$ for $\rho\in (0,1]$. $\widetilde x^*(t)\in \rm int(\mathbb X)$ since $\mathbb X$ is convex. It follows that
\[\begin{array}{rl}
 & \|\varphi(x^*(t),t)-\widetilde x^*(t))\|^2_2\\
 \noalign{\vskip6pt}
 = & \|\varphi(x^*(t),t)-x^*(t)+x^*(t)-(1-\rho)x^*(t)-\rho x^0\|^2_2\\
 \noalign{\vskip6pt}
 = & \|\varphi(x^*(t),t)-x^*(t)-\rho(x^0-x^*(t))\|^2_2\\
 \noalign{\vskip6pt}
 = & \|\varphi(x^*(t),t)-x^*(t)\|^2_2 -2\rho(x^0-x^*(t))^\top(\varphi(x^*(t),t)-x^*(t))+\rho^2\|x^*(t)-x^0\|^2.
\end{array}
\]
Recall that \((x^0-x^*(t))^\top(\varphi(x^*(t),t)-x^*(t))> 0.\) Thus, when $\rho>0$ is sufficiently small, one must have
\[\|\varphi(x^*(t),t)-\widetilde x^*(t))\|^2_2< \|\varphi(x^*(t),t)-x^*(t)\|^2_2.\]
Therefore, $\psi(\varphi(x^*(t),t))\ne x^*(t)$. Hence, \[(x^0-x^*(t))^\top(\varphi(x^*(t),t)-x^*(t))\le 0.\] 
A contradiction occurs. This completes the proof.
\end{proof}
Theorem \ref{th0} says that for any given $t\in [0,1]$, there exists an $x\in\mathbb X$ such that $\varphi(x,t)=x$. This conclusion shows that the homotopy system (\ref{s1}) is precisely the same as the system (\ref{oc}) and the set of solutions to the homotopy system can be rewritten as 
\[H^{-1}=\{(x,t)\in \mathbb X\times [0,1]\,|\,x=\varphi(x,t)\}.\] 
For our further development, we need the following fixed point theorem.
\begin{theorem}[\bfseries{Browder's Fixed Point Theorem}]\label{th1}
Let $S$ be a non-empty,  compact and convex subset of $\mathbb{R}^{n} $ and let $f: S \times [0,1] \rightarrow S$ be a continuous correspondence. Then, the set $G=\{(s,t)\in S \times [0,1]\,|\,s=f(s,t)\}$ contains a connected subset $G^{\rm c}$ such that $(S \times \{1\})\bigcap G^{\rm c}\ne \emptyset$ and $(S \times \{0\})\bigcap G^c\ne \emptyset$.
\end{theorem}
As a corollary of Theorem \ref{th1}, we come to the following result.
\begin{corollary}\label{co3}
$H^{-1}$ contains a connected subset $H_c^{-1}$ such that $\mathbb X\times \{0\}\cap H_c^{-1}\ne \emptyset$ and $\mathbb X\times \{1\}\cap H_c^{-1}\ne \emptyset$.
\end{corollary}
Corollary~\ref{co3} assures the global convergence of the GRSAA differentiable homotopy method.

\subsection{A GRSAA Differentiable Homotopy Path}
To develop an efficient GRSAA differentiable homotopy method for computing a solution to the SAA (\ref{ori}), we need to construct an everywhere smooth path that starts from $(x^0,1)$ and ends at a solution to the SAA (\ref{ori}) on the level of $t=0$. A common sufficient condition for the existence of such a smooth path is that zero is a regular value of a homotopy system. The next lemma shows that zero is a regular value of $h(x,1)$ at the starting point $(x^0,1)$.
\begin{lemma}\label{lem_2}
	At $t=1$, zero is a regular value of $h(x,1)$ on $\mathbb{X}$.
\end{lemma}
\begin{proof}
Taking derivatives of $h(x,1)$ with respect to $x$, we obtain that the Jacobian matrix of $h(x,1)$  is an identity matrix. Thus zero is a regular value of $h(x,1)$. This completes the proof.
\end{proof}
We prove in the following theorem that, under the condition that zero is  a regular value of $h(x,t)=0$, the set $H^{-1}$ contains an everywhere smooth path leading to a solution to (\ref{ori}) as $t$ goes to zero. 
\begin{theorem}\label{conver}
	Suppose that zero is a regular value of $h(x,t)=0$ on $\mathbb R^n\times(0,1)$. Then there exists an everywhere smooth path in $H^{-1}$, which starts from the unique solution $(x^0,1)$ on the level of $t = 1$ and ends at a solution to (\ref{ori}) on the level of $t = 0$.
\end{theorem}
\begin{proof}
 Corollary \ref{co3} tells us that $H^{-1}$ has a connected component that intersects both $\mathbb R^n\times\{1\}$ and $\mathbb R^n\times\{0\}$.  Since the system (\ref{s1}) has a unique solution $(x^0, 1)$ at $t=1$, the  connected component in $H^{-1}$ is unique. The condition of Theorem~\ref{conver} together with the well-known implicit function theorem ensures that  the connected component forms a smooth path, which starts from $(x^0,1)$ and ends at a solution to (\ref{ori}) on the level of $t=0$.
\end{proof}
This theorem relies on the condition that zero is a regular value of $h(x,t)$ on $\mathbb R^n\times(0,1)$. To get rid of this condition, a general approach is to subtract from $h(x,t)$ a perturbation term of $t(1-t)\alpha$, where $\alpha\in\mathbb{R}^n$ and $\|\alpha\|$ is sufficiently small. Subtracting from $h(x,t)$ the perturbation term, we get
$h(x,t;\alpha)=h(x, t) -t(1-t)\alpha = 0$. For any fixed $\alpha$, let $h_\alpha(x, t) = h(x,t)-t(1-t)\alpha$ and $H^{-1}_\alpha=\{(x,t)\in\mathbb{R}^n\times[0,1]\,|\,h_\alpha(x,t)=0\}$. As $t=0$ and $t=1$, the perturbation term disappears and $h(x,t)=h_\alpha(x,t)$. Clearly, as $\|\alpha\|$ goes to zero, the distance between $H^{-1}$ and $H^{-1}_\alpha$ approaches zero. Therefore, as $\|\alpha\|$ is sufficiently small, $H^{-1}_\alpha$ also contains a connected component intersecting both $\mathbb X\times \{0\}$ and $\mathbb X\times\{1\}$. These results together lead to the following theorem.

\begin{theorem}\label{gth}
	For generic $\alpha\in\mathbb{R}^n$ with sufficiently small $\|\alpha\|$, $H^{-1}_\alpha$ contains a unique smooth path starting from $(x^0,1)$.
\end{theorem}
\begin{proof}
	As a mapping of $(x, t, \alpha)$ on $\mathbb{R}^n\times(0,1)\times\mathbb{R}^n$, the Jacobian matrix of $h(x,t;\alpha)$ is an $n \times (2n+1)$ matrix given by $[D_{x,t}h(x, t)\; -t(1- t)I_{n}]$, where $I_n$ is an $n\times n$ identity matrix. Therefore, the Jacobian matrix of $h(x,t;\alpha)$ has full row rank on $\mathbb{R}^n\times(0,1)\times\mathbb{R}^n$. This together with Lemma~\ref{lem_2} asserts that zero is a regular value of $h(x,t;\alpha)$ on $\mathbb{R}^n\times(0,1]\times\mathbb{R}^n$. We know from the well-known transversality theorem \cite{eaves1999general} that, for almost all $\alpha\in\mathbb{R}^n$, zero is also a regular value of $h_\alpha(x,t)$. Therefore, when $\|\alpha\|$ is sufficiently small, we derive  from the well-known implicit function theorem that the connected component in $H^{-1}_\alpha$ intersecting $\mathbb X\times\{1\}$ forms a unique smooth path starting from $(x^0,1)$. This completes the proof.
	\end{proof}
  Theorem~\ref{gth} ensures that one can follow the smooth path in $H^{-1}_\alpha$ starting from $(x^0,1)$ to find a solution to the SAA~(\ref{ori}). In our numerical implementation of the proposed method, we always set $\alpha=0\in\mathbb{R}^n$.

\section{A Gradually Reinforced SAA Simplicial Homotopy Method}\label{sec4}
To further evince the advantage of the GRSAA differentiable homotopy method,  
we describe in this section a gradually reinforced SAA (GRSAA) simplicial homotopy method. As an underlying triangulation of the method, we make use of the $D_2$-triangulation of $(0,1]\times\mathbb{R}^n$ with continuous refinement of grid size, whose definition and pivot rules are given in \cite{Dang1993}.
 Given $\tau_0>0$, let $t_\ell=1/(1 + \tau_0 \ell)$, $\ell=0,1,\ldots$. Then, $t_\ell$, $\ell=0,1,\ldots$, form a descent sequence with $t_0=1$ and $\lim\limits_{\ell\rightarrow\infty}t_\ell=0$. Given $\tau_1>0$, we define $q_\ell=\tau_1\ell$, $\ell=1,2, \ldots$. Therefore, as $\ell\rightarrow \infty$, $q_\ell\rightarrow\infty$.
Given the initial grid size $\varpi_0>0$ and $r_\ell\in\{1/j\;|\;j=1,2,\ldots\}$, let $\varpi_{\ell+1}=r_\ell\varpi_\ell$,  $\ell=0,1,\ldots$, which is the gird size of a simplex of the $D_2$-triangulation in $\{t_\ell\}\times\mathbb{R}^n$. Let $x^0\in\mathbb X$ be a given point. We still utilize the continuously differentiable function $\theta_\ell (t)$ defined in Section \ref{sec2} to form a simplicial homotopy system, 
\begin{equation}\label{sh}
\begin{array}{l}
g(t,x):=(1-t) \bar d(t,x)+t(x-x^0)=0,
\end{array}
\end{equation}
where $\bar d(t,x)=\theta(t)\dfrac{1}{q_\ell}\sum\limits_{i=1}^{q_\ell} f(x,\xi_i)+(1-\theta(t))\dfrac{1}{q_{\ell+1}}\sum\limits_{i=1}^{q_{\ell+1}} f(x,\xi_i)$ for $t_{\ell+1}\le t\le t_\ell$. Clearly, when $t=1$, the system (\ref{sh}) has a unique solution $x^0$. As follows, we exploit the homotopy mapping $ g(t,x)$ to attain a GRSAA simplicial homotopy method.

Let $v^c=(0,x^0-\frac{\varpi_0}{n+1}e)$ with $e=(1,1,\ldots,1)^{\top}\in\mathbb{R}^{n}$.
Let $D_2$ be the collection of all simplices of the $D_2$-triangulation of $(0,1]\times\mathbb{R}^n$ after the translation of $v^c$, where the translation ensures that $(1, x^0)$ is in the interior of a unique simplex in $\{1\}\times\mathbb{R}^n$. We denote a  $q$-dimensional simplex in $D_2$ by $\sigma=\langle v^0,v^1,\ldots,v^{q}\rangle$, where $v^i$ is a vertex of $\sigma$ for $i=0,1,\ldots,q$.
\begin{definition}\label{d2}
 $\sigma=\langle v^0,v^1,\ldots,v^{n}\rangle\in D_2$ is a complete simplex if the system,{\small
 \begin{equation}\label{ls1}\sum\limits_{l=0}^{n}\zeta_l
 \left(\begin{array}{c}
 1\\
g(v^l)
 \end{array}\right)  =\left(\begin{array}{c}
 1\\
 0
 \end{array}\right)\text{ and }\zeta\ge 0,
 \end{equation}}
 has a solution.
 \end{definition}

\begin{assumption}[Nondegeneracy Assumption]\label{a2}
The system~(\ref{ls1}) has a unique solution with $\zeta>0$.\footnote{Note that this assumption can be eliminated if the lexicographic pivoting rule in Eaves~\cite{eaves1971linear} and Todd~\cite{todd1976} is applied in the linear programming step of the following method.}
\end{assumption}
Following a similar argument to that in \cite{chen2019differentiable}, we know that
there is a unique complete simplex contained in $\{1\}\times\mathbb{R}^n$. Given these notations, a simplicial homotopy method for computing an approximate solution to the SSE can be stated as follows.
 \begin{description}
 \item[Step 0:] Let $\delta_0$ be a given sufficiently small positive number. Let $\tau_0=\langle v^0,v^1,\ldots,v^{n}\rangle$ be the unique complete simplex in $\{1\}\times\mathbb{R}^n$ with $v^0=(1,x^0)$ and $\sigma_0$ be the unique $(n+1)$-dimensional simplex in $[t_1,1]\times \mathbb{R}^n$ with $\tau_0$ as its facet. Let $v^+$ be the vertex of $\sigma_0$ opposite to $\tau_0$. Let $\ell=0$ and go to Step 1.
 \item[Step 1:]Perform a linear programming step to bring $\left(\begin{array}{c}
 1\\
 g(v^+)
 \end{array}\right)$ into the system,
 \(\sum\limits_{l=0}^{n}\zeta_l
 \left(\begin{array}{c}
 1\\
 g(v^l)
 \end{array}\right)
  =\left(\begin{array}{c}
 1\\
 0
 \end{array}\right).
 \)
Suppose that $\left(\begin{array}{c}
 1\\
 g(v^l)
 \end{array}\right)$ leaves the system. Let $\tau_{\ell+1}$ be the facet of $\sigma_\ell$ opposite to $v^l$ and $v^l=v^+$.  Go to Step 2.
\item[Step 2:]  If $\tau_{\ell+1}\subset \{t_{\ell+1}\}\times\mathbb{R}^n$ and $t_{\ell+1}\le\delta_0$, the method terminates. Otherwise, let $\sigma_{\ell+1}$ be the unique simplex that shares together with $\sigma_\ell$ a common facet $\tau_{\ell+1}$. Let $v^+$ be the vertex of $\sigma_{\ell+1}$ opposite to $\tau_{\ell+1}$ and $\ell=\ell+1$. Go to Step 1.
\end{description}
As follows, we show that the simplicial path generated by the GRSAA simplicial homotopy method leads to a solution to the SSE (\ref{s0}). Let
\( g^{-1}(0)=\{(t,x)\in [0,1]\times \mathbb{R}^n\;|\;g(t,x)=0\}.\)  Then it follows from {\bf Assumption~\ref{a1}} that $g^{-1}(0)$ is a compact set. As a result of  Definition \ref{d2} and the continuity of $f$, it is easy to verify that all the complete simplices are contained in a bounded set of $[0,1]\times \mathbb{R}^n$. Under the nondegeneracy assumption, following the same argument as that in Todd~\cite{todd1976}, one can derive that all the simplices generated by the GRSAA simplicial homotopy method are different, that is, no cycle will occur. Let $\tau_
\ell$, $\ell=1,2,\ldots$, be the complete simplices generated by the method. Let $\bar t_\ell$ be the smallest value in $(0,1]$ such that $\tau_\ell\subset (0, \bar t_\ell]\times \mathbb{R}^n$. Since all $\tau_
\ell$, $\ell=1,2,\ldots$, are contained in a bounded set, we must have $\bar t_\ell\rightarrow 0$ as $\ell\rightarrow\infty$.

Let $\sigma=\{t\}\times \langle x^0,x^1,\ldots,x^{n}\rangle$ be an $n$-dimensional simplex in $\{t\}\times \mathbb{R}^n$. For $y=\sum\limits_{l=0}^{n}\zeta_l(y^l)$ with $\zeta_l\ge 0$ and $\sum\limits_{l=0}^{n}\zeta_l=1$, a piecewise linear approximation of $g(t,y)$ is given by
\(\bar g(t,y)=\sum\limits_{l=0}^{n}\zeta_l g(t,y^l).\)
 \begin{lemma}\label{as1}
 For $t\in (0, 1]$, let $\sigma^*=\{t\}\times \langle y^0,y^1,\ldots,y^{n}\rangle$ be a complete simplex in $\{t\}\times \mathbb{R}^n$ with $y^l\in \mathbb{R}^n$. Let $\zeta^*$ be the corresponding unique solution of the system~(\ref{ls1}). Then, $(y^a(t))=\sum\limits_{l=0}^{n}\zeta^*_ly^l\in \mathbb{R}^n$ is a zero point of $\bar g(t,\cdot)$.
 \end{lemma}
 \begin{proof}
 It follows from the system~(\ref{ls1})  that
 $\sum\limits_{l=0}^{n}\zeta^*_l g(t,y^l)=0$.
  Thus,
 $\bar g(t,y^a(t))=\sum\limits_{l=0}^{n}\zeta^*_l g(t,y^l)=0$.
 Therefore, $y^a(t)\in \mathbb{R}^n$ is a zero point of $\bar g(t,\cdot)$.
 \end{proof}

 Since $g(t,y)$ is uniformly continuous on $[0,1]\times \mathbb{R}^n$,  there is a constant $\rho_0>0$ such that $\|g(t,y)-g(\hat t,\hat y)\|\le\rho_0 \|(t,y)-(\hat t,\hat y)\|$ for any $(t,y)$ and $(\hat t,\hat y)$ in $[0,1]\times \mathbb{R}^n$.
 \begin{lemma} \label{as2} For $t\in (0, 1]$, let $y^a(t)\in \mathbb{R}^n$ be a zero point of $\bar g(t,\cdot)$.
   Then,
 $\|g(t,y^a(t))\|\le\rho_0 \varpi(t)$, where $\varpi(t)$ is the gird size of the triangulation restricted on $\{t\}\times \mathbb{R}^n $.
 \end{lemma}
 \begin{proof}  Assume that $(t,y^a(t))\in\sigma^*=\{t\}\times \langle y^0,y^1,\ldots,y^{n}\rangle$. Then, {\small\[\begin{array}{rl}
  &\|g(t,y^a(t))\| = \|g(t,y^a(t))-\bar g(t,y^a(t))\|
  = \|\sum\limits_{l=0}^{n}\zeta^*_l (g(t,y^a(t))-g(t,y^l)) \| \\
  
  \le & \sum\limits_{l=0}^{n}\zeta^*_l\|g(t,y^a(t))- g(t,y^l)\| \le \sum\limits_{l=0}^{n}\zeta^*_l\rho_0\|y^a(t)-y^l\|
  \le  \rho_0 \varpi(t).
  \end{array}\]}
  The proof is completed.
 \end{proof}
This lemma implies that every limit point of the simplicial path yields a solution to the SSE (\ref{s0}) with probability one as $t$ goes to zero.

\section{Numerical Performance}\label{sec5}
We apply in this section the GRSAA differentiable homotopy method to solve several applications of the SSE (\ref{s0}), which include a stochastic market equilibrium problem and a stochastic variational inequality problem. To numerically trace the smooth path, we employ a standard predictor-corrector method.\footnote{Interested readers can refer to \cite{allgower2012numerical,eaves1999general} for more details about the predictor-corrector method.} Moreover, we have made numerical comparisons of the GRSAA differentiable homotopy method with the GRSAA simplicial homotopy method and a standard  differentiable homotopy method. All these methods are coded in MatLab(R2019a). The computation has been carried out on a  2.00 GHz Windows PC with CORE i7.  The numerical results further confirm the effectiveness and efficiency of the GRSAA differentiable homotopy method. 

\subsection{A Stochastic Market Equilibrium Problem}
This subsection is concerned with a stochastic market equilibrium problem.  Suppose that there are three goods and two firms in an economy. The consumers in the economy can be represented by one agent, who has a constant elasticity of substitution (CES) utility function with a stochastic substitution parameter, $u(x,y,z)=(2x^\xi+3y^\xi+z^\xi)^{1/\xi}$, where $x$, $y$ and $z$ are the amounts of the three goods being consumed. The initial endowment is given by $w= (1,1,1)^\top$. In the economy, the agent wants to maximize her utility by determining an optimal consumption plan. A direct application of the sufficient and necessary optimality conditions to the convex optimization problem for the agent yields the excess demand of the agent at the market price $p=(p_x,p_y,p_z)^\top$, 
\begin{equation*}
\begin{array}{l}
f(p,\xi)= p^\top w\left(\dfrac{p_1}{g_1}, \dfrac{p_2}{g_2}, \dfrac{p_3}{g_3}\right)^\top,
\end{array}
\end{equation*} 
where $g_1=p_1^\xi+\frac{2}{3}^{\frac{1}{\xi-1}}p_2^\xi+ 2^{\frac{1}{\xi-1}} p_3^{\xi}$, $g_2= \frac{3}{2}^{\frac{1}{\xi-1}} p_1^\xi+p_2^\xi+ 3^{\frac{1}{\xi-1}}p_3^\xi$, and $g_3= \frac{1}{2}^{\frac{1}{\xi-1}}p_1^\xi+\frac{1}{3}^{\frac{1}{\xi-1}}p_2^\xi+ p_3^\xi$ with $p_1 =p_x^{\frac{1}{\xi-1}}$, $p_2 =p_y^{\frac{1}{\xi-1}}$, and $p_3 =p_z^{\frac{1}{\xi-1}}$. The production technology of firms are described by the following matrix,
\[A =
\begin{bmatrix}
-\dfrac{3}{2}&1&1\\
-1&-\dfrac{77}{27}&\dfrac{11}{9}
\end{bmatrix}.\]
Since no firm makes a positive profit in an equilibrium, we have the constraint $Ap\le 0$.  Assume that $p_x+p_y+p_z\le 1$.\footnote{This constraint  ensures that the feasible set is a compact polytope.} Then the feasible set reads as $P = \{p\in\mathbb R^3_+\,:\, Ap\le 0,\,e^\top p \le 1\}$, where $e=(1,1,1)^\top$. We say an equilibrium is reached if and only if  there exist vectors $z\in \mathbb R^6$ and $s\in\mathbb R^6$ together with $p$ satisfying that $\mathbb E_\xi f(p,\xi)-B^\top z = 0$, $B p + s- b=0$, and $zs = 0$, where $B = [A;\;-I_n;\;e^\top]\in\mathbb R^{6\times 3}$ and $b = (0,0,0,0,0,1)^\top$ \cite{van1994adjustment,zhan2020differentiable}.  

The corresponding GRSAA differentiable homotopy method for this problem is as follows. We set the sample size $N=10^4$. After randomly generating a batch of samples $\xi_1, \xi_2, \ldots, \xi_{N}$ of the stochastic variable $\xi$ from the uniform distribution on $[-1,1]$, one can approximate the expected value $\mathbb E_\xi f(p,\xi)$ by an SAA scheme, $\dfrac{1}{N}\sum\limits_{i=1}^{N}f(p,\xi_i)$. Let $t_0=1$, $t_L=0$ and $t_\ell$, $\ell=1,2,\ldots,L-1$, be randomly generated from $(0, 1)$ with $t_\ell<t_{\ell-1}$. We make up the following unconstrained convex optimization problem,
\begin{equation}\label{mre}
\begin{array}{ll}
\max\limits_{x\in\mathbb R^3}&(1-t)x^\top(d(p,t)+t^{\kappa_0} \sum\limits_{i=1}^6\text{log}(b_i-B_i x))-\dfrac{t}{2}\|x-p^0\|^2,\\
\end{array}
\end{equation}
where $\kappa_0\ge 2$, $p^0$ is any given interior point in $P$ and $d(p,t)=d^\ell(p,t)$ for $t\in[t_\ell,t_{\ell-1}]$ with $d^\ell(p,t)$ as  defined in Section \ref{sec2}.  An application of the optimality conditions to the problem~(\ref{mre}) together with a fixed point argument leads to the following GRSAA differentiable homotopy system,
\begin{equation}\label{mes}
\begin{array}{l}
(1-t)(d(p,t)-B^\top z(y))-t(p-p^0) = 0,\\
\noalign{\vskip 6pt}
Bp+s(y)-b = 0,
\end{array}
\end{equation}
where  $z(y)=\left(\dfrac{\sqrt{y^2+4t}-y}{2}\right)^{\kappa_0}$ and $s(y)=\left(\dfrac{\sqrt{y^2+4t}+y}{2}\right)^{\kappa_0}$ with $y\in \mathbb R^6$ being a new variable.\footnote{A well-chosen transformation of variables can significantly reduce the number of variables and constraints so that numerical efficiency can be greatly improved. This technique has been frequently used in the literature such as \cite{chen2019differentiable,herings2006computing}.}  Therefore, there exists a smooth path contained in the set of solutions to the system (\ref{mes}), which starts from a unique solution on the level of $t=1$ and ends at an approximate equilibrium for the original market equilibrium problem on the level of $t=0$. By applying a predictor-corrector method to trace the smooth path, we eventually reach a solution $p=(0.40, 0.45, 0.15)$ in 12 iterations and 0.6096 seconds.  Figure \ref{f11} shows the distances from the current point to the desired solution of the original problem at each iteration for the GRSAA differentiable homotopy method. The changes of different variables in iterations are illustrated in Figure \ref{f12}.
\begin{figure}[!htbp]
\centering
\includegraphics[height=5cm,width=9cm]{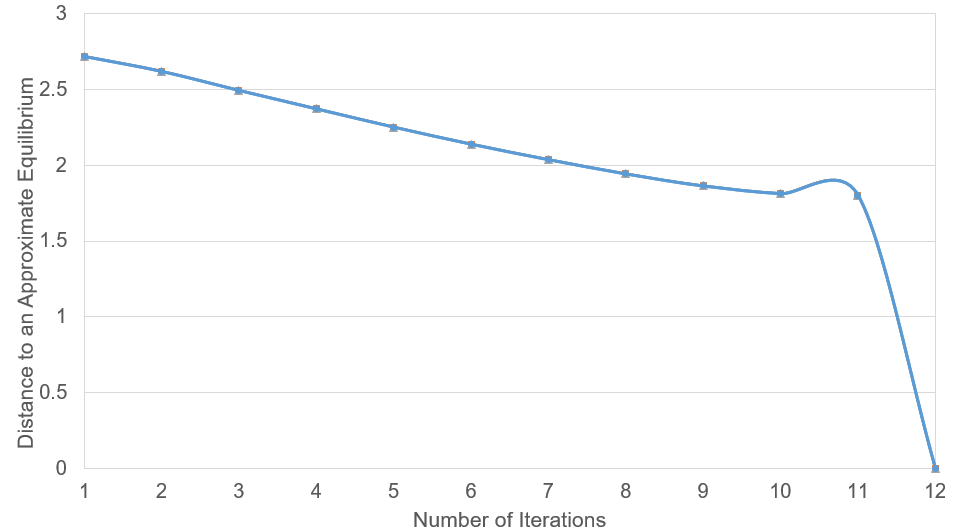}
 \caption{Numerical Results for the Market Equilibrium Problem}\label{f11}
\end{figure}

\begin{figure}[!htbp]
\centering
\includegraphics[height=7cm,width=9cm]{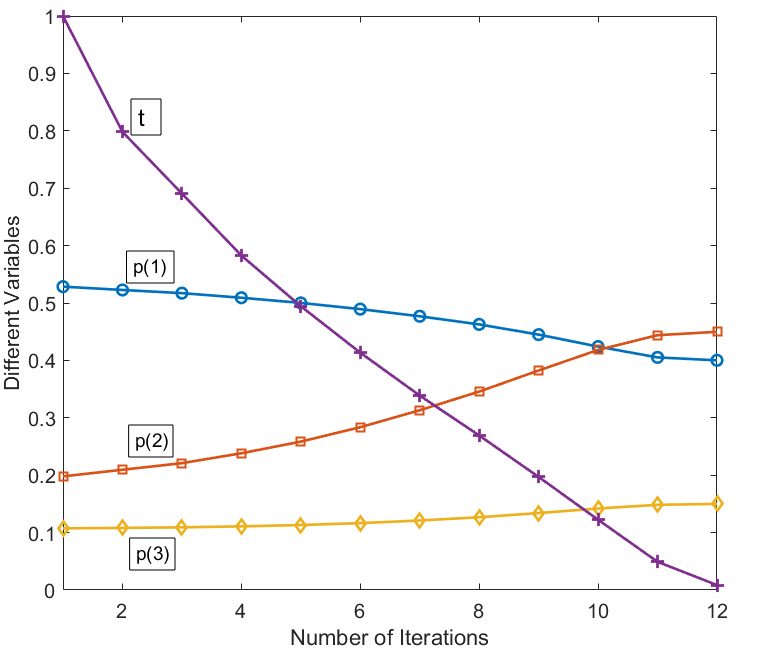}
 \caption{Changes of Variables in Iterations}\label{f12}
\end{figure}

\subsection{A Stochastic System of Equations}

A simplicial homotopy method was proposed in \cite{dang1991d1} to compute a solution to a deterministic system of equations, $x_i-5\sin(i\sum\limits_{j=1}^nx_j)=0$, $i=1,2,\ldots,n$. This subsection focuses on a stochastic version of this problem, $\mathbb E_\xi [f(x,\xi)]=0$, where the mapping $f\,:\,\mathbb R^n\rightarrow \mathbb R^n$ is given by 
\begin{equation}\label{exm2}
\begin{array}{l}
f_i(x,\xi)=x_i-5\sin(i\sum\limits_{j=1}^nx_j+\xi),
\end{array}
\end{equation}
$i=1, 2, \ldots, n$.  
We employ the GRSAA differentiable homotopy method to solve the SSE (\ref{exm2}) under different $n\in\{3, 4, 5, 6, 7, 8, 9, 10\}$. For numerical comparisons, we also apply the GRSAA simplicial homotopy method to solve the same problems.  In the implementation of the GRSAA simplicial homotopy method, we set the initial grid size of the $D_2$-triangulation $\bar\omega_0=1$, the factor to refine the grid size $r_\ell=0.5$ for all $\ell$, $t_0=1$, $t_1=0.5$, $t_\ell = 1/(1+7000 \ell)$ for $\ell = 2, 3, \ldots$.   Additionally, we choose the sample size $q_\ell=500\ell$ for $\ell = 1, 2, \ldots$. Clearly,  $t_\ell\rightarrow 0$ and $q_{\ell}\rightarrow\infty$ as $\ell\rightarrow \infty$. Moreover, to make the comparisons more convincing, $t_\ell$, $\ell=0,1,\ldots$, for the GRSAA differentiable homotopy method are consistent with those for the GRSAA simplicial homotopy method and both methods start from the same randomly generated starting point.   Each case with different $n$ is tested for 20 times and the average computational costs are reported in Table \ref{t21}, where ``ITER'' is the average number of iterations, ``TIME'' is the average computational time (in seconds), ``GRSAA-DH'' and ``GRSAA-SH''  represent respectively the GRSAA differentiable homotopy method and GRSAA simplicial homotopy method, and ``RATIO'' stands for the ratio of the numerical results of GRSAA-DH to GRSAA-SH.\footnote{Note that if  the computational time exceeds 1000 and the number of iterations is larger than $10^7$, we record the value as ``INF''.}
\begin{table}[!htbp]
\centering
\caption{Average Computational Cost of Two Methods}\label{t21}
\renewcommand\arraystretch{1.3}
\setlength{\tabcolsep}{4.5mm}{
\begin{tabular}{|c|c|c|c|c|c|c|}
\hline
$n$&\multicolumn{2}{c|}{\textbf{GRSAA-DH}} & \multicolumn{2}{c|}{GRSAA-SH}& \multicolumn{2}{c|}{RATIO (\%)}\\
\hline
&ITER&TIME&ITER&TIME&ITER&TIME\\
\hline
3&\textbf{78}&\textbf{2.72}&1181&2.58&6.60&105.42\\
\hline
4&\textbf{96}&\textbf{3.76}&3707&11.30&2.59&33.27\\
\hline
5&\textbf{98}&\textbf{4.71}&5738&12.88&1.71&36.56\\
\hline
6&\textbf{105}&\textbf{5.49}&14797&32.65&0.71&16.81\\
\hline
7&\textbf{108}&\textbf{6.54}&34011&59.83&0.32&10.93\\
\hline
8&\textbf{123}&\textbf{8.53}&78253&99.11&0.16&8.61\\
\hline
9&\textbf{129}&\textbf{9.45}&202596&198.30&0.06&4.76\\
\hline
10&\textbf{272}&\textbf{16.33}&INF&INF&-&-\\
\hline
\end{tabular}}
\end{table}

From the columns of Table \ref{t21}, we find that both the average number of iterations (ITER) and the average computational time (TIME) become greater and greater with the increasing of $n$ for the two methods. This result coincides with our intuition, that is, the larger the problem is, the harder it is for the methods to solve.  From the rows of Table \ref{t21}, we observe that the  GRSAA differentiable homotopy method significantly outperforms the GRSAA simplicial homotopy method both in the number of iterations and computational time when $n>3$.  The advantage of the GRSAA-DH method over the GRSAA-SH method becomes more remarkable when the number of variables $n$ is higher, which implies that the GRSAA-DH method is relatively less sensitive to $n$. 

Figure \ref{f21} presents the changes of the maximal, average and minimal computational time in $n$ among the 20 tests for the two methods. Figure \ref{f22} illustrates the changes of the maximal, average and minimal number of iterations in $n$ for the two methods.
\begin{figure}[!htbp]
\includegraphics[height=0.4\textwidth]{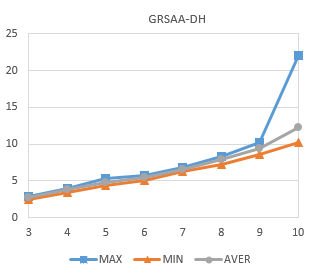}\quad
\includegraphics[height=0.4\textwidth]{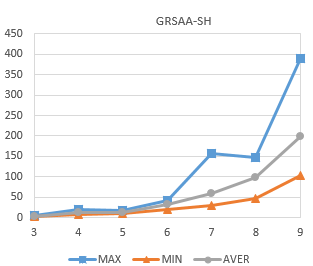}
 \caption{Computational Time for GRSAA-DH and GRSAA-SH}\label{f21}
\end{figure}
\begin{figure}[!htbp]
\includegraphics[height=0.4\textwidth]{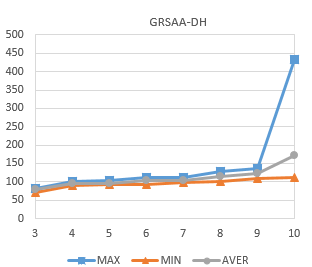}\quad
\includegraphics[height=0.4\textwidth]{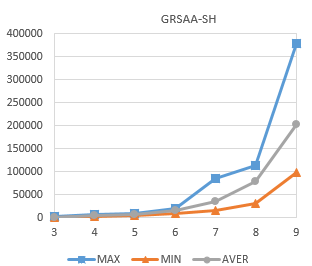}
 \caption{Number of Iterations for GRSAA-DH and GRSAA-SH}\label{f22}
\end{figure}

Regarding the GRSAA-DH method, one can observe from Figure \ref{f21} and Figure \ref{f22} that  as $n\le 9$, the difference between the maximal computational time and the minimal computational time is fairly small. As $n=10$, there are some extreme values that enlarge the gap between the maximal value and the minimal value, but the number of such extreme cases are quite few, since the average computational time is just about 12 seconds, which is near the minimal computational time of 10 seconds. This indicates that the computational time for most tests is between 10 and 12 seconds. However, for the GRSAA-SH method, when $n>6$, one can easily see that the gaps among the maximal, average and minimal computational time are much larger. Similar phenomena can be found for the number of iterations. These numerical results show that the GRSAA-DH method performs much more stable than the GRSAA-SH method. 

To further demonstrate the advantage of the GRSAA-DH method, we have used the method to solve several large-scale SSEs, for which the GRSAA-SH method fails to find a reasonable approximate solution after the computational time is over 5000 seconds. We have run the GRSAA-DH method on each test 10 times. The numerical results are reported in Table \ref{t22} and Figure \ref{f25}, which once again certify that the GRSAA differentiable homotopy method is able to effectively and efficiently solve larger-scale problems.
\begin{table}[!htbp]
\centering
\caption{Computational Time of GRSAA-DH for Large-Scale Cases}\label{t22}
\renewcommand\arraystretch{1.4}
\setlength{\tabcolsep}{7mm}{
\begin{tabular}{c|c|c|c}
\hline
$n$&MAX&MIN&\textbf{AVER}\\
\hline
15&55.99&19.65&\textbf{36.389}\\
\hline
16&476.85&52.27&\textbf{219.771}\\
\hline
17&774.83&102.06&\textbf{416.047}\\
\hline
18&1367.73&110.96&\textbf{559.361}\\
\hline
19&1786.92&244.85&\textbf{729.847}\\
\hline
20&1247.82&464.09&\textbf{821.583}\\
\hline
21&1571.14&325.18&\textbf{985.259}\\
\hline
22&1811.67&380.60&\textbf{1230.041}\\
\hline
23&2133.52&597.04&\textbf{1317.084}\\
\hline
24&2214.6&911.96&\textbf{1396.252}\\
\hline
25&2826.06&1214.31&\textbf{1479.845}\\
\hline
\end{tabular}}
\end{table}
\begin{figure}[!htbp]
\centering
\includegraphics[height=0.4\textwidth]{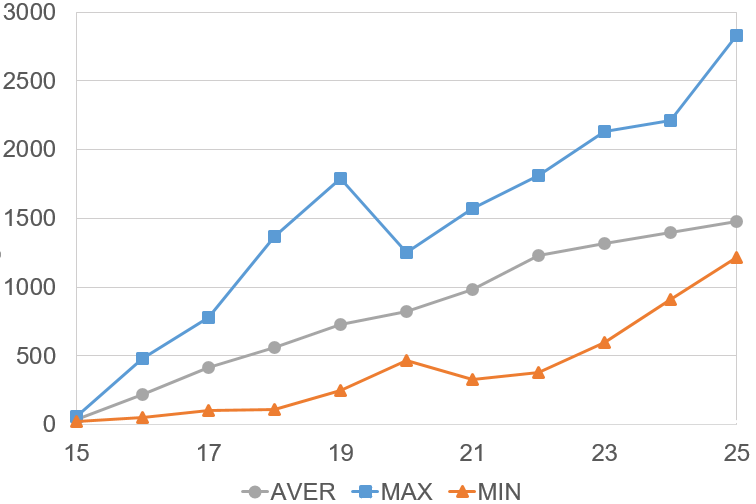}
 \caption{Computational Time for GRSAA-DH for Different Cases}\label{f25}
\end{figure}

\subsection{A Stochastic Variational Inequality}
This subsection intends to solve a stochastic variational inequality (SVI) problem, 
\begin{equation}\label{exm31}
\begin{array}{ll}
(y-x)^\top \mathbb E_\xi[f(x,\xi)]\le 0,&\forall\,y\in \mathbb Y,
\end{array}
\end{equation}
where $f\,:\,\mathbb R^n\rightarrow \mathbb R^n$ is a continuously differentiable mapping with $f_i(x,\xi)=x_i-e^{\text{cos}(i\sum_{j=1}^nx_j+\xi)}$ for $i=1,2,\ldots,n$, and $\mathbb Y=\{y=(y_1,\ldots,y_n)^\top\in\mathbb R^n\,|\,-10\le y_i\le 10,\, \forall \, i=1,2,\ldots,n\}$. To get a GRSAA differentiable homotopy system for the SVI problem, we constitute the optimization problem,
\begin{equation*}
\begin{array}{ll}
\max\limits_{y\in\mathbb R^n} &y^\top\mathbb E[f(x,\xi)]\\
\text{s.t.}&Cy\le b,
\end{array}
\end{equation*}
where $C=[I_n\;-I_n]^\top \in\mathbb R^{2n\times n}$ and $b=10(e^\top,-e^\top)^\top\in\mathbb R^{2n}$ with $I_n\in\mathbb R^{n\times n}$ being an identity matrix and $e=(1,1,\ldots,1)^\top\in\mathbb R^n$. Applying the optimality conditions to this optimization problem, we derive with a fixed point argument  an equivalent problem to  the SVI problem~(\ref{exm31}),
\begin{equation}\label{exm32}
\begin{array}{l}
\mathbb E[f(x,\xi)]-C^\top\lambda=0,\\
\noalign{\vskip 6pt}
Cx-b+z=0,\\
\noalign{\vskip 6pt}
\lambda z=0,\quad \lambda\ge 0,\quad z\ge 0.
\end{array}
\end{equation} 
We have generated numerous samples for the stochastic variable $\xi$ and applied the SAA, $\dfrac{1}{N}\sum\limits_{i=1}^N f(x,\xi_i)$, to estimate the expected value $\mathbb E[f(x,\xi)]$. The corresponding GRSAA differentiable homotopy system to the SSE (\ref{exm32}) is as follows:
\begin{equation}\label{exm33}
\begin{array}{l}
\mathbb (1-t)(d(x,t)-C^\top\lambda)-t(x-x^0)=0,\\
\noalign{\vskip 6pt}
Cx-b+z=0,\\
\noalign{\vskip 6pt}
\lambda z=t^{\kappa_0},
\end{array}
\end{equation}
where $x^0$ is an interior point in the compact convex set $\mathbb Y$ and $d(x,t)$ is defined as in Section \ref{sec2}.  After a transformation of variables, that is, 
$\lambda_i(w)=\left(\dfrac{\sqrt{w_i^2+4t}+w_i}{2}\right)^{\kappa_0}$ and $\quad z_i(w)= \left(\dfrac{\sqrt{w_i^2+4t}-w_i}{2}\right)^{\kappa_0}$ for $i=1,2,\ldots,2n$,
the third group of equations of the system (\ref{exm33}) vanishes and the system (\ref{exm33}) becomes an equivalent form,
\begin{equation}\label{exm34}
\begin{array}{l}
\mathbb (1-t)(d(x,t)-C^\top\lambda(w))-t(x-x^0)=0,\\
\noalign{\vskip 6pt}
Cx-b+z(w)=0.
\end{array}
\end{equation}

Clearly, the efficiency of the proposed method depends on  the number of variables $n$, the sample size $N$, and the number of divisions $L$. Next, we study the impact of these factors on the performance of the GRSAA differentiable homotopy method.
\begin{itemize}
\item First, the GRSAA-DH method has been used to solve the system (\ref{exm34}) under different pairs of $(N,n)$, where $N\in\{10^4, 10^5, 10^6\}$ and $n\in \{1, 2, 3\}$. The number of divisions for the method is fixed to be $L=N/2$. To make a numerical comparison,  we have also employed a standard differentiable homotopy method, which can be considered as a special case of the GRSAA-DH method without a gradual reinforcement process, that is, $L=1$, to solve the same problems. Each case has been run 10 times and the average computational time is reported in Table \ref{t31}.
\begin{table}[!htbp]
\centering
\caption{Numerical Results for the Two Methods}\label{t31}
\renewcommand\arraystretch{1}
\setlength{\tabcolsep}{2mm}{
\begin{tabular}{ccc}
\hline
&GRSAA-DH& Stardard Homotopy\\
\hline
\diagbox{$n$}{$N=10^4$}&\\
\hline
1&\textbf{1.03}&1.18\\
2&\textbf{3.94}&5.64\\
3&\textbf{6.57}&10.98\\
\hline
\diagbox{$n$}{$N=10^5$}&\\
\hline
1&\textbf{8.61}&10.95\\
2&\textbf{35.61}&49.32\\
3&\textbf{61.91}&100.35\\
\hline
\diagbox{$n$}{$N=10^6$}&\\
\hline
1&\textbf{80.20}&95.55\\
2&\textbf{358.47}&474.51\\
3&\textbf{619.35}&1052.27\\
\hline
\end{tabular}}
\end{table}

One can see from Table \ref{t31} that the computational costs of the two methods become larger and larger  with the growing of sample size and the number of variables. For each fixed $n$, the computational time of the GRSAA-DH method increases approximately linearly with $N$. For example, when $n=2$ and $N$ changes from $10^4$ to $10^5$, the computational time of the GRSAA-DH method increases from 3.94 to 35.61. Comparing the last two columns of Table \ref{t31}, one can observe that the GRSAA-DH method is more competitive than the standard  differentiable homotopy method, which verifies that the gradual reinforcement process indeed makes a significant difference. 
\item Second, we want to explore how the computational time is influenced by the number of divisions $L$ and determine an appropriate value of $L$ for the GRSAA differentiable homotopy method. In theory, $L$ can be chosen as any value that is not larger than the sample size $N$. Nonetheless,  it is intuitive that a too large or too small value of $L$ may cause a low numerical efficiency, which has been partially acknowledged in Table \ref{t31}, since the standard  differentiable homotopy method is same as the GRSAA-DH method with $L=1$. In order to achieve a better efficiency, it is necessary to find a suitable value of $L$ when implementing the  GRSAA-DH method. However, finding an optimal value of $L$ is an NP-hard problem in theory and can only be realized through numerical experiments. We have implemented the GRSAA-DH method with different choices of $L$ to solve the system (\ref{exm34}) and plotted the changes of the computational time in various values of $L$ in Figures \ref{f31}, \ref{f32} and \ref{f33}. In these experiments, $N\in\{10^4, 10^5, 10^6\}$ and $n\in \{1, 2, 3\}$.
\end{itemize}
\begin{figure}[!htbp]
\centering
\includegraphics[height=0.35\textwidth]{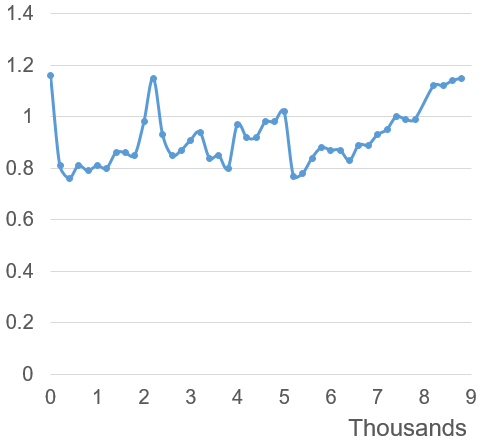}\quad
\includegraphics[height=0.35\textwidth]{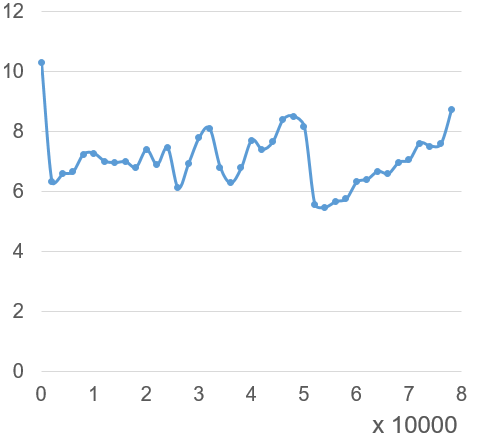}\quad
\includegraphics[height=0.35\textwidth]{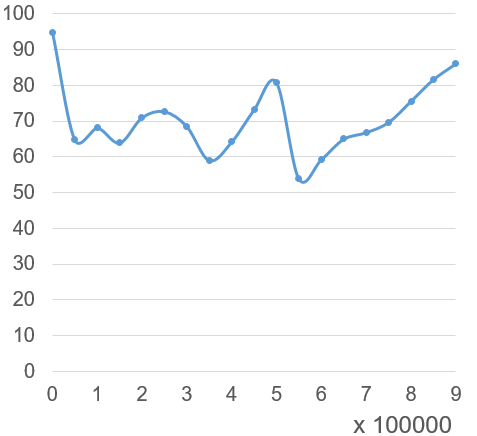}
 \caption{$n=1$ and $N=10^4, 10^5, 10^6$}\label{f31}
\end{figure}
\begin{figure}[!htbp]
\centering
\includegraphics[height=0.35\textwidth]{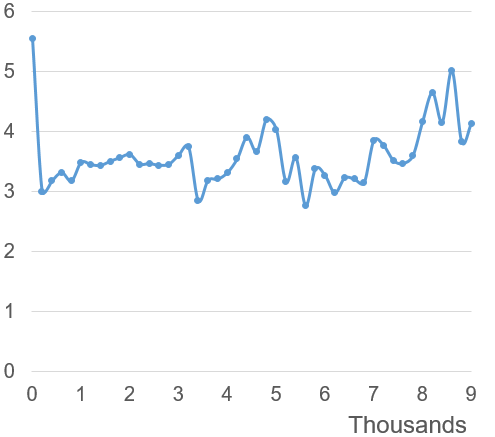}\quad
\includegraphics[height=0.35\textwidth]{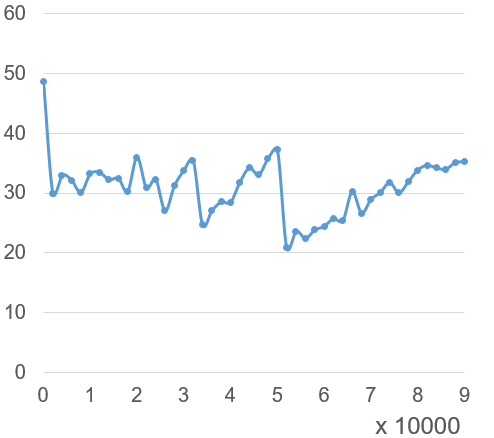}\,
\includegraphics[height=0.35\textwidth]{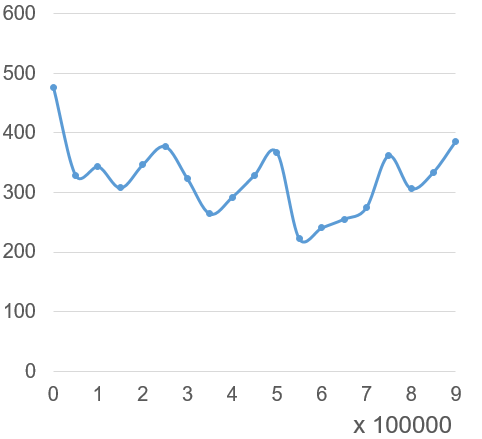}
 \caption{$n=2$ and $N=10^4, 10^5, 10^6$}\label{f32}
\end{figure}
\begin{figure}[!htbp]
\centering
\includegraphics[height=0.35\textwidth]{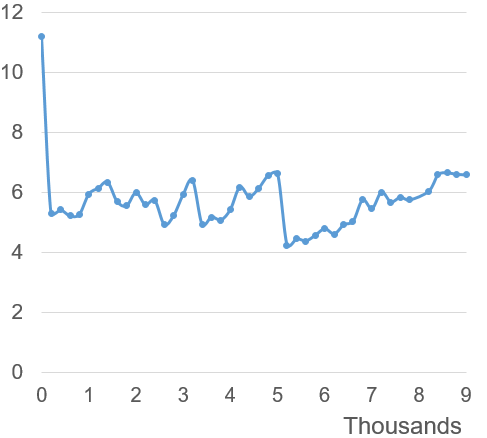}\quad
\includegraphics[height=0.35\textwidth]{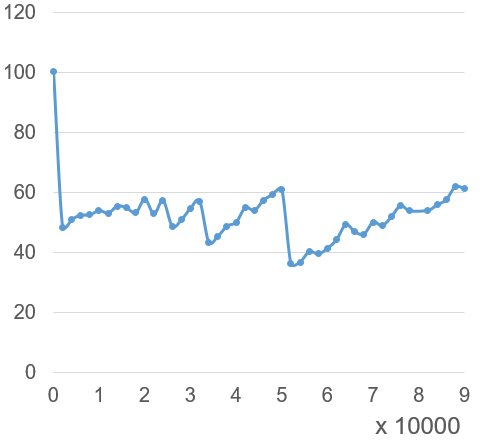}\quad
\includegraphics[height=0.35\textwidth]{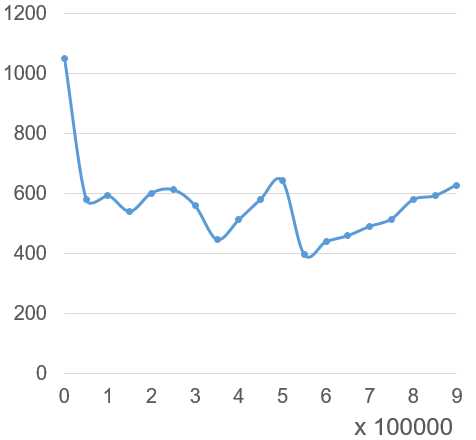}
 \caption{$n=3$ and $N=10^4, 10^5, 10^6$}\label{f33}
\end{figure}

From Figures~\ref{f31}, \ref{f32} and \ref{f33}, we find that the computational time presents a highly  analogous trend for  different cases especially with the relatively high number of variables and a large batch of samples. More specifically, the computational time always decreases as $L=1$ and reaches a local minimum about $L= 0.55 N$. Notice that the local minimum also has a sufficiently competitive performance over the entire possible choices of $L$. Hence, when applying the GRSAA-DH method, one can roughly determine the ``best'' number of divisions according to $L=0.55 N$, which surly performs much better than the standard  differentiable homotopy method.

\section{Conclusions}\label{sec6}

This paper has exploited the sample-average-approximation (SAA) scheme to approximate a stochastic system of equations (SSE) and developed a gradually reinforced sample-average-approximation (GRSAA) differentiable homotopy method to solve the SSE with very large sample size. By introducing a sequence of continuously differentiable functions of the homotopy  parameter ranging between zero and one, we have established a  continuously differentiable homotopy system, which is able to gradually increase the sample size as the homotopy  parameter decreases. The solution set of the homotopy system contains an everywhere smooth path, which starts from an arbitrarily given point and ends at a satisfactory approximate solution to the original SSE. 

The GRSAA differentiable homotopy method provides an effective link between a standard differentiable homotopy method and a gradually reinforced SAA scheme. The proposed method is able to greatly reduce the computational cost for a solution to the SSE with large sample size and retain the inherent property of global convergence with a standard homotopy method. To make numerical comparisons, we have carried out extensive numerical tests. The numerical results further confirm that two main features of the GRSAA differentiable homotopy method,  gradual reinforcement in sample size and  differentiability, can significantly enhance numerical effectiveness and efficiency.

\bibliographystyle{spmpsci}

\bibliography{references}

\begin{thebibliography}{10}
\providecommand{\url}[1]{{#1}}
\providecommand{\urlprefix}{URL }
\expandafter\ifx\csname urlstyle\endcsname\relax
  \providecommand{\doi}[1]{DOI~\discretionary{}{}{}#1}\else
  \providecommand{\doi}{DOI~\discretionary{}{}{}\begingroup
  \urlstyle{rm}\Url}\fi

\bibitem{allgower2012numerical}
Allgower, E.L., Georg, K.: Numerical continuation methods: an introduction,
  vol.~13.
\newblock Springer Science \& Business Media (2012)

\bibitem{awoniyi1983efficient}
Awoniyi, S.A., Todd, M.J.: An efficient simplicial algorithm for computing a
  zero of a convex union of smooth functions.
\newblock Mathematical Programming \textbf{25}(1), 83--108 (1983)

\bibitem{Brown1996aComputing}
Brown, D.J., Demarzo, P.M., Eaves, B.C.: Computing equilibria when asset
  markets are incomplete.
\newblock Econometrica \textbf{64}(1), 1--27 (1996)

\bibitem{Byrd2012Sample}
Byrd, R.H., Chin, G.M., Nocedal, J., Wu, Y.: Sample size selection in
  optimization methods for machine learning.
\newblock Mathematical Programming \textbf{134}(1), 127--155 (2012)

\bibitem{chen2022stochastic}
Chen, L., Liu, Y., Yang, X., Zhang, J.: Stochastic approximation methods for
  the two-stage stochastic linear complementarity problem.
\newblock SIAM Journal on Optimization \textbf{32}(3), 2129--2155 (2022)

\bibitem{chen2022dynamic}
Chen, X., Shen, J.: Dynamic stochastic variational inequalities and convergence
  of discrete approximation.
\newblock SIAM Journal on Optimization \textbf{32}(4), 2909--2937 (2022)

\bibitem{chen2012stochastic}
Chen, X., Wets, R.J.B., Zhang, Y.: Stochastic variational inequalities:
  residual minimization smoothing sample average approximations.
\newblock SIAM Journal on Optimization \textbf{22}(2), 649--673 (2012)

\bibitem{chen2019differentiable}
Chen, Y., Dang, C.: A differentiable homotopy method to compute perfect
  equilibria.
\newblock Mathematical Programming pp. 1--33 (2019)

\bibitem{dang1991d1}
Dang, C.: The ${D}_1$-triangulation of rn for simplicial algorithms for
  computing solutions of nonlinear equations.
\newblock Mathematics of Operations Research \textbf{16}(1), 148--161 (1991)

\bibitem{Dang1993}
Dang, C.: The ${D}_2$-triangulation for simplicial homotopy algorithms for
  computing solutions of nonlinear equations.
\newblock Mathematical Programming \textbf{59}(1-3), 307--324 (1993)

\bibitem{dang2020an}
Dang, C., Herings, P.J.J., Li, P.: An interior-point differentiable
  path-following method to compute stationary equilibria in stochastic games.
\newblock INFORMS Journal on Computing \textbf{34}(3), 1403--1418 (2022)

\bibitem{duchi2011adaptive}
Duchi, J., Hazan, E., Singer, Y.: Adaptive subgradient methods for online
  learning and stochastic optimization.
\newblock Journal of Machine Learning Research \textbf{12}(7) (2011)

\bibitem{eaves1971linear}
Eaves, B.C.: The linear complementarity problem.
\newblock Management Science \textbf{17}(9), 612--634 (1971)

\bibitem{B1972Homotopies}
Eaves, B.C.: Homotopies for computation of fixed points.
\newblock Mathematical Programming \textbf{3}(1), 1--22 (1972)

\bibitem{eaves1999general}
Eaves, B.C., Schmedders, K.: General equilibrium models and homotopy methods.
\newblock Journal of Economic Dynamics and Control \textbf{23}(9), 1249--1279
  (1999)

\bibitem{eibelshauser2023logarithmic}
Eibelsh{\"a}user, S., Klockmann, V., Poensgen, D., von Schenk, A.: The
  logarithmic stochastic tracing procedure: A homotopy method to compute
  stationary equilibria of stochastic games.
\newblock INFORMS Journal on Computing \textbf{35}(6), 1511--1526 (2023)

\bibitem{facchinei2007finite}
Facchinei, F., Pang, J.S.: Finite-dimensional variational inequalities and
  complementarity problems.
\newblock Springer Science \& Business Media (2007)

\bibitem{ferris2013complementarity}
Ferris, M.C., Mangasarian, O.L., Pang, J.S.: Complementarity: Applications,
  algorithms and extensions, vol.~50.
\newblock Springer Science \& Business Media (2013)

\bibitem{G1999Sample}
Gürkan, G., Ozge, A.Y., Robinson, S.M.: Sample-path solution of stochastic
  variational inequalities.
\newblock Mathematical Programming \textbf{84}(2), 313--333 (1999)

\bibitem{herings2000two}
Herings, P.J.J.: Two simple proofs of the feasibility of the linear tracing
  procedure.
\newblock Economic Theory \textbf{15}(2), 485--490 (2000)

\bibitem{herings2002computing}
Herings, P.J.J., Kubler, F.: Computing equilibria in finance economies.
\newblock Mathematics of operations research \textbf{27}(4), 637--646 (2002)

\bibitem{herings2001differentiable}
Herings, P.J.J., Peeters, R.: A differentiable homotopy to compute {N}ash
  equilibria of $n$-person games.
\newblock Economic Theory \textbf{18}(1), 159--185 (2001)

\bibitem{herings2010homotopy}
Herings, P.J.J., Peeters, R.: Homotopy methods to compute equilibria in game
  theory.
\newblock Economic Theory \textbf{42}(1), 119--156 (2010)

\bibitem{herings2004stationary}
Herings, P.J.J., Peeters, R.J.: Stationary equilibria in stochastic games:
  Structure, selection, and computation.
\newblock Journal of Economic Theory \textbf{118}, 32--60 (2004)

\bibitem{herings2006computing}
Herings, P.J.J., Schmedders, K.: Computing equilibria in finance economies with
  incomplete markets and transaction costs.
\newblock Economic Theory \textbf{27}(3), 493--512 (2006)

\bibitem{hu2012sample}
Hu, J., Homem-de Mello, T., Mehrotra, S.: Sample average approximation of
  stochastic dominance constrained programs.
\newblock Mathematical Programming \textbf{133}(1-2), 171--201 (2012)

\bibitem{hu2020sample}
Hu, Y., Chen, X., He, N.: Sample complexity of sample average approximation for
  conditional stochastic optimization.
\newblock SIAM Journal on Optimization \textbf{30}(3), 2103--2133 (2020)

\bibitem{huang2023finding}
Huang, Z.H., Li, Y.F., Miao, X.: Finding the least element of a nonnegative
  solution set of a class of polynomial inequalities.
\newblock SIAM Journal on Matrix Analysis and Applications \textbf{44}(2),
  530--558 (2023)

\bibitem{iusem2019incremental}
Iusem, A.N., Jofr{\'e}, A., Thompson, P.: Incremental constraint projection
  methods for monotone stochastic variational inequalities.
\newblock Mathematics of Operations Research \textbf{44}(1), 236--263 (2019)

\bibitem{Jiang2008Stochastic}
Jiang, H., Xu, H.: Stochastic approximation approaches to the stochastic
  variational inequality problem.
\newblock IEEE Transactions on Automatic Control \textbf{53}(6), 1462--1475
  (2008)

\bibitem{jiang2022convergence}
Jiang, J., Sun, H., Zhou, B.: Convergence analysis of sample average
  approximation for a class of stochastic nonlinear complementarity problems:
  from two-stage to multistage.
\newblock Numerical Algorithms \textbf{89}(1), 167--194 (2022)

\bibitem{judd2012finding}
Judd, K.L., Renner, P., Schmedders, K.: Finding all pure-strategy equilibria in
  games with continuous strategies.
\newblock Quantitative Economics \textbf{3}(2), 289--331 (2012)

\bibitem{kellogg1976constructive}
Kellogg, R.B., Li, T.Y., Yorke, J.: A constructive proof of the brouwer
  fixed-point theorem and computational results.
\newblock SIAM Journal on Numerical Analysis \textbf{13}(4), 473--483 (1976)

\bibitem{kleywegt2002sample}
Kleywegt, A.J., Shapiro, A., Homem-de Mello, T.: The sample average
  approximation method for stochastic discrete optimization.
\newblock SIAM Journal on Optimization \textbf{12}(2), 479--502 (2002)

\bibitem{kubler2000computing}
Kubler, F., Schmedders, K.: Computing equilibria in stochastic finance
  economies.
\newblock Computational Economics \textbf{15}(1-2), 145--172 (2000)

\bibitem{lee2023polyhedral}
Lee, K., Tang, X.: On the polyhedral homotopy method for solving generalized
  nash equilibrium problems of polynomials.
\newblock Journal of Scientific Computing \textbf{95}(1), 13 (2023)

\bibitem{Li2020}
Li, P., Dang, C.: An arbitrary starting tracing procedure for computing subgame
  perfect equilibria.
\newblock Journal of Optimization Theory and Applications \textbf{186}(2),
  667--687 (2020)

\bibitem{merrill1972applications}
MERRILL, O.H.: Applications and extensions of an algorithm that computes fixed
  points ofcertain upper semi-continuous point to set mappings.
\newblock Ph.D. thesis, University of Michigan (1972)

\bibitem{na2023inequality}
Na, S., Anitescu, M., Kolar, M.: Inequality constrained stochastic nonlinear
  optimization via active-set sequential quadratic programming.
\newblock Mathematical Programming pp. 1--75 (2023)

\bibitem{nemirovski2009robust}
Nemirovski, A., Juditsky, A., Lan, G., Shapiro, A.: Robust stochastic
  approximation approach to stochastic programming.
\newblock SIAM Journal on Optimization \textbf{19}(4), 1574--1609 (2009)

\bibitem{Pflug1996Optimization}
Pflug, G.C.: Optimization of Stochastic Models.
\newblock Springer US (1996)

\bibitem{rockafellar2019solving}
Rockafellar, R.T., Sun, J.: Solving monotone stochastic variational
  inequalities and complementarity problems by progressive hedging.
\newblock Mathematical Programming \textbf{174}(1-2), 453--471 (2019)

\bibitem{rockafellar2009variational}
Rockafellar, R.T., Wets, R.J.B.: Variational analysis, vol. 317.
\newblock Springer Science \& Business Media (2009)

\bibitem{Scarf1967The}
Scarf, H.: The approximation of fixed points of a continuous mapping.
\newblock SIAM Journal on Applied Mathematics \textbf{15}(5), 1328--1343 (1967)

\bibitem{schmedders1998computing}
Schmedders, K.: Computing equilibria in the general equilibrium model with
  incomplete asset markets.
\newblock Journal of Economic Dynamics and Control \textbf{22}(8-9), 1375--1401
  (1998)

\bibitem{shapiro1998simulation}
Shapiro, A., Homem-de Mello, T.: A simulation-based approach to two-stage
  stochastic programming with recourse.
\newblock Mathematical Programming \textbf{81}(3), 301--325 (1998)

\bibitem{shapiro2000rate}
Shapiro, A., Homem-de Mello, T.: On the rate of convergence of optimal
  solutions of {M}onte {C}arlo approximations of stochastic programs.
\newblock SIAM Journal on Optimization \textbf{11}(1), 70--86 (2000)

\bibitem{todd1976}
Todd, M.J.: The computation of fixed points and applications.
\newblock Lecture Notes in Economics and Mathematical Systems, Springer, Berlin
  \textbf{124} (1976)

\bibitem{van1994adjustment}
Van Den~Elzen, A., Van Der~Laan, G., Talman, D.: An adjustment process for an
  economy with linear production technologies.
\newblock Mathematics of operations research \textbf{19}(2), 341--351 (1994)

\bibitem{wang2020efficient}
Wang, G., Wei, X., Yu, B., Xu, L.: An efficient proximal block coordinate
  homotopy method for large-scale sparse least squares problems.
\newblock SIAM Journal on Scientific Computing \textbf{42}(1), A395--A423
  (2020)

\bibitem{zhan2020smooth}
Zhan, Y., Dang, C.: A smooth homotopy method for incomplete markets.
\newblock Mathematical Programming \textbf{190}(1-2), 585--613 (2021)

\bibitem{zhan2020differentiable}
Zhan, Y., Li, P., Dang, C.: A differentiable path-following algorithm for
  computing perfect stationary points.
\newblock Computational Optimization and Applications \textbf{76}, 571--588
  (2020)

\bibitem{zhang2023solution}
Zhang, J., Xiao, Q., Li, L.: Solution space exploration of low-thrust
  minimum-time trajectory optimization by combining two homotopies.
\newblock Automatica \textbf{148}, 110798 (2023)

\bibitem{zhou2014smoothing}
Zhou, Z., Yu, B.: A smoothing homotopy method for variational inequality
  problems on polyhedral convex sets.
\newblock Journal of Global Optimization \textbf{58}(1), 151--168 (2014)

\end{thebibliography}

\end{document}